\documentclass{amsart}
\usepackage{amsmath}
\usepackage{framed,comment,enumerate}
\usepackage[pdftex]{graphicx}
\usepackage{epstopdf}
\usepackage{xcolor, framed}

\def\avint{\mathop{\,\rlap{-}\!\!\int}\nolimits}

\def\R {\mathbb{R}}
\def\eps{\varepsilon}

\def\OP{obstacle problem }
\def\FB{free boundary }
\def\FN{fully nonlinear }
\def\AP{Alt-Phillips }

\def\Sph{\mathbb{S}^{d-1}}

\def\PosS{\{u>0\}}
\def\ZeroS{\{u=0\}}

\def\FDu{F(D^2u)}
\def\uga{u^{\gamma-1}}

\newtheorem{prop}{Proposition}[section]
\newtheorem{thm}{Theorem}[section]

\newtheorem{cor}{Corollary}[section]
\newtheorem{lem}{Lemma}[section]
\theoremstyle{definition}
\newtheorem{defi}{Definition}[section]
\newtheorem{rem}{Remark}[section]
\numberwithin{equation}{section}

\title[Fully nonlinear Alt-Phillips]{On the fully nonlinear Alt-Phillips equation } 

\author{Yijing Wu}
\address{Department of Mathematics, University of Maryland, College Park, MD, USA}
\email{yijingwu@umd.edu}

\author{Hui Yu}
\address{Department of Mathematics,	Columbia University, New York, USA}
\email{ huiyu@math.columbia.edu}
\thanks{HY is supported by  NSF grant DMS-1954363.}

\begin{document}

\begin{abstract}
For a parameter $\gamma\in(1,2)$, we study the \FN version of the  Alt-Phillips equation, $\FDu=u^{\gamma-1}$, for $u\ge 0.$ We establish the optimal regularity of the solution, as well as the $C^1$ regularity of the regular part of the free boundary. 

\end{abstract}

\maketitle

\section{Introduction}
Given a domain, $\Omega$, in the $d$-dimensional Euclidean space $\R^d$, and a parameter, $\gamma$,  from the interval $(0,2)$,  the \textit{classical \AP equation} is 
\begin{equation}\label{SecondEquation}
\begin{cases}
\Delta u=\uga& \\
u\ge 0&
\end{cases}\text{ in $\Omega$.}
\end{equation} When $\gamma=1$, the first equation is understood as $$\Delta u=\chi_{\{u>0\}},$$where $\chi_E$ denotes the characteristic function of a set $E$.

Equation \eqref{SecondEquation} appears in the study of  gas distribution in a porous catalyst pellet \cite{A}. The solution $u$ models the density of the gas, and the right-hand side $\uga$ denotes the rate of reaction with the catalyst. In this context, the relevant range of $\gamma$ is $(1,2)$.   Mathematically, however, it is interesting to study the regularity of the solution as well as the \textit{\FB}$\partial\PosS$ for the entire range of $\gamma\in(0,2)$.  

For such parameters, this equation was first studied by Phillips \cite{P} and Alt-Phillips \cite{AP} as the Euler-Lagrange equation of the \textit{\AP functional} $$u\mapsto \int_\Omega |\nabla u|^2+ \frac{2}{\gamma}u^\gamma.$$
 Notably, this embeds two of the most well-studied \FB problems, the \textit{Bernoulli problem} and the \textit{obstacle problem}, as special cases. When $\gamma\to0$, this energy degenerates to the Alt-Caffarelli functional \cite{AC}
$$u\mapsto \int_\Omega|\nabla u|^2+\chi_{\PosS},$$which is the underlying energy in the Bernoulli problem.  When $\gamma=1$, the \AP functional becomes the energy in the  obstacle problem $$u\mapsto \int_\Omega|\nabla u|^2+2u.$$ 

Both the Bernoulli problem and the obstacle problem have been extensively studied in the past few decades. For the classic theory on these two problems, see, for instance,  \cite{AC, C1, C2, CS, PSU}. For some exciting recent developments, see \cite{CSV, D, ESV, FiSe}. 

To compare \eqref{SecondEquation} with  the Bernoulli problem and the obstacle problem, the starting point is its scaling property. 

Define a new parameter $$\beta=\frac{2}{2-\gamma}\in(1,+\infty).$$  If a function $u$ solves \eqref{SecondEquation} in $B_1$, then its rescaling $$u_r(x):=\frac{u(rx)}{r^\beta}$$ is a solution in $B_{1/r}$. This suggests that the solution grows like $|x|^\beta$ from the free boundary.  Now noting that solutions to the Bernoulli problem grow linearly, and that solutions to the obstacle problem grow quadratically, it is natural to study the \textit{distorted solutions} of \eqref{SecondEquation},  $u^{1/\beta}$ and $u^{2/\beta}$, so that techniques developed for these two classic problems can be applied.

With this simple yet fundamental observation, a lot of works were devoted to the study of the classical \AP equation \eqref{SecondEquation}.

Concerning the regularity of the solution,  Alt-Phillips showed that the minimizer\footnote{For $\gamma\in(0,1)$, there is no uniqueness of solution to \eqref{SecondEquation}. The regularity theory is developed for minimizers of the Alt-Phillips functional.} is in $C^{1,\beta-1}$ if $\gamma\in (0,1)$ \cite{AP}. For $\gamma\in (1,2)$, the solution is in $C^{2,\alpha}$ for $\alpha=\min\{\gamma-1,\delta\}$, where $\delta>0$ is a dimensional constant \cite{P}.  These are optimal for the corresponding range of $\gamma.$

They also obtained results about the free boundary $\partial\PosS$. When $\gamma\in(0,1)$, the \FB  has finite $(d-1)$-dimensional Hausdorff measure. Recall that $d$ is the dimension of the ambient space. Moreover, the \FB is locally a $C^{1,\alpha}$-hypersurface outside a set of zero $(d-1)$-dimensional Hausdorff measure \cite{AP, W}.  

When $\gamma\in(1,2)$, the \FB decomposes into the \textit{regular part} and the \textit{singular part}. A point $x_0\in\partial\PosS$ is a \textit{regular point} if 
\begin{equation}\label{ThirdEquation}\limsup_{r\to0}\frac{|B_r(x_0)\cap\{u=0\}|}{r^d}>0.\end{equation}Otherwise, it is a \textit{singular point}. Based on results in \cite{AP}, Bonorino showed that the regular part is locally a $C^{1,\alpha}$-hypersurface, and the singular part is rectifiable \cite{B}.

Some of these results have been extended to a general class of equations by De Silva-Savin \cite{DS}. The two-phase problem was studied by Fotouhi-Shahgholian \cite{FoSh}.

Recently, there has been quite some interest in  \FB problems involving \textit{\FN operators}. Apart from various industrial applications, these problems often lead to unexpected mathematical discoveries of deep regularizing mechanisms.   See, for instance, Caffarelli-Duque-Vivas \cite{CDV}, Figalli-Shahgholian \cite{FiSh}, Indrei-Minne \cite{IM} and Ros-Oton-Serra \cite{RS}.

Despite these exciting developments, very little is known about the\textit{ fully nonlinear Alt-Phillips equation} 
\begin{equation}\label{FirstEquation}
\begin{cases}
\FDu=\uga &\\
u\ge0 &
\end{cases} \text{ in $\Omega$.}
\end{equation}  Here $\gamma\in(0,2)$, and $F$ is a convex  elliptic operator. Solutions are understood in the viscosity sense. Details about such operators and viscosity solutions are given in the next section. On a slightly different note, the singular version of \eqref{FirstEquation} when $\gamma\in(0,1)$ has been studied by Ara\'ujo-Teixeira in \cite{AT}. The case for $\gamma=0$ has been studied in the context of singular perturbations by Ricarte-Teixeira in \cite{RT}.

For the range $\gamma\in(0,2)$, the only case that has been addressed is when $\gamma=1$, that is, the \textit{fully nonlinear obstacle problem} 
\begin{equation}\label{FNObstacle}
\begin{cases}
\FDu=\chi_{\PosS} &\\
u\ge0 &
\end{cases} \text{ in $\Omega$.}
\end{equation}For this special case, the regularity of the solution and the regular part of the free boundary were studied by Lee \cite{L}. Together with Savin, the second author obtained regularity of  the singular part of the free boundary  \cite{SY}. Intersection of the free boundary with the fixed boundary was studied  by Indrei \cite{I}.

For general $\gamma\in(0,1)\cup(1,2),$ regularity of the solution as well as the \FB remains completely open. 

Unlike its classical counterpart \eqref{SecondEquation}, we do not have fundamental tools like monotonicity formulae developed by Spruck \cite{S} and Weiss \cite{W2}. In particular, this means that blow-ups are in general not homogeneous, and their \textit{contact sets} $\{u=0\}$ are not cones. Moreover, due to the nonlinearity of the operator, the distorted solutions $u^{1/\beta}$ and $u^{2/\beta}$ no longer solve clean equations. 

Compared with the fully nonlinear obstacle problem \eqref{FNObstacle}, the difficulty lies in the nontrivial right-hand side of \eqref{FirstEquation}. In \eqref{FNObstacle}, the right-hand side is constant in the \textit{non-contact set} $\PosS$. In particular,  we can  get equations for  the first and second derivatives of the solution by differentiating \eqref{FNObstacle}. This is no longer the case for \eqref{FirstEquation}. Here the right-hand side $\uga$ blows up near the \FB when $\gamma\in(0,1)$, and degenerates when $\gamma\in(1,2)$. In either case, it is not differentiable near $\partial\PosS.$

Dealing with these challenges requires new ideas. 

In this work, we develop techniques to study  the fully nonlinear \AP equation \eqref{FirstEquation} for $\gamma\in(1,2).$ We establish the optimal regularity of the solution $u$, and the $C^1$-regularity of the regular part of the free boundary. 

Before introducing our main results, we remark that the study of the singular part of the \FB is far from complete, even for the classical problem \eqref{SecondEquation}. Also, we only study \eqref{FirstEquation} for $\gamma\in(1,2)$. As mentioned at the beginning of this introduction, this is the range relevant to the physical model. The study of \eqref{FirstEquation} for $\gamma\in(0,1)$ requires a different set of ideas, and is postponed to a future work. 

Now we state our main results.

The starting point is a Harnack-type estimate for \eqref{FirstEquation}, which says that the solution at a point controls its values in an entire neighborhood. 

\begin{thm}\label{HarnackThm}
Suppose that $F$ is a  convex uniformly elliptic operator with $F(0)=0.$  For $\gamma\in(1,2)$, let $u$ be a solution to \eqref{FirstEquation} in $B_R$ for some $R>0$. 

Then $$\sup_{B_{R/2}}u\le C(\inf_{B_{R/2}}u+R^\beta)$$ for $\beta=\frac{2}{2-\gamma}$ and a universal constant $C$. 
\end{thm} 

A constant is called universal if it only depends on the dimension $d$, the ellipticity constant of the operator, and the parameter $\gamma$. 

Such an estimate for the linear problem \eqref{SecondEquation}  was established by \AP\cite{AP}.

The reader can find several consequences of this estimate in Section 3. Now we note that it leads  to the following universal regularity estimate on the solution.

\begin{thm}\label{FirstResult}
Under the same assumptions as in Theorem \ref{HarnackThm},  the solution to \eqref{FirstEquation} is in $C^{2,\alpha}_{loc}(\Omega)$ for a universal $\alpha\in(0,1).$ 

Moreover, if $B_1\subset\Omega$, then $$\|u\|_{C^{2,\alpha}(B_{1/2})}\le C(u(0)+1)$$ for a universal  constant $C$. 
\end{thm} Qualitatively, this result is a direct consequence of the Evans-Krylov theorem. To get the universal bound, however, we need Theorem \ref{HarnackThm}.  

This $C^{2,\alpha}$-regularity of the solution is optimal just like in the classical case.

Theorem \ref{HarnackThm}  also gives universal control over several scaling invariant quantities, i.e.,  $|D^2u/\uga|$.   This allows us to perform a blow-up analysis near the free boundary. 

Classification of blow-up profiles turns out to be challenging. Even when the operator is the Laplacian, this classification is far from complete \cite{BBLT, FoSh}. However, if the free boundary point is a regular point, that is, a point satisfying \eqref{ThirdEquation}, we can get enough geometric information to obtain the following result on \FB regularity. 

\begin{thm}\label{SecondResult}
Together with the assumptions as in Theorem \ref{HarnackThm}, we further assume that $F$ is either differentiable at $0$ or that $F$ is homogeneous.  

Suppose that $u$ solves \eqref{FirstEquation}, then the regular part of the \FB is relatively open in $\partial\PosS$, and is locally a $C^1$-hypersurface. 
\end{thm}

This article is structured as follows: In the next section, we gather some necessary preliminaries about fully nonlinear operators and our problem \eqref{FirstEquation}.  In Section 3, we prove the Harnack-type estimate, Theorem \ref{HarnackThm}, and some immediate consequences, including Theorem \ref{FirstResult}. In Section 4, we study properties of blow-up profiles at a regular free boundary point. Finally, these properties are used to prove Theorem \ref{SecondResult} in Section 5.

\section{Preliminaries and notations}
This section is divided into two subsections. In the first, we collect some results on fully nonlinear elliptic operators. The main reference is Caffarelli-Cabr\'e \cite{CC}. In the second subsection, we state some definitions related to the \FN \AP problem, and give a few intermediate properties of the solution. 
\subsection{Fully nonlinear elliptic operators}
Let $\mathcal{S}_d$ denote the space of $d$-by-$d$ symmetric matrices. Let $\Lambda$ be a constant in $[1,+\infty)$. 

A function 
$$F:\mathcal{S}_d\to\R$$ 
is a \textit{uniformly elliptic operator with ellipticity constant $\Lambda$} if it satisfies 
\begin{equation}\label{Ellipticity}\frac{1}{\Lambda}\|P\|\le F(M+P)-F(M)\le\Lambda\|P\|\end{equation} for all $M,P\in\mathcal{S}_d$ and $P\ge 0.$

Solutions to these operators are understood in the \textit{viscosity} sense. For the definition of a viscosity solution, see Definition 2.3 in  \cite{CC}. Soon enough we see that the solution to our problem is actually a classical solution. 

Viscosity solutions have the following stability property. See Proposition 4.11 in \cite{CC}.
\begin{prop}\label{Stability}
Let $F_h$ be a sequence of \FN elliptic operators with the same ellipticity constant $\Lambda$. 

Let $u_h\in C(\Omega)$ be viscosity solutions to $$F_h(D^2u_h)=f_h \text{ in $\Omega$}$$ for a sequence of continuous functions $f_h.$

Suppose that $F_h\to F$ locally uniformly in $\mathcal{S}_d$, and that $u_h\to u$ and $f_h\to f$ locally uniformly in $\Omega$, then $$\FDu=f \text{ in $\Omega$}.$$\end{prop} 

On top of its ellipticity \eqref{Ellipticity}, we assume that $F$ satisfies 
\begin{equation}\label{Convexity}
\begin{cases}
&F \text{ is convex, }\\ &F(0)=0, \\&\text{the trace operator is a sub-differential of } F \text{ at } 0.
\end{cases}
\end{equation} 

\begin{rem}Convexity and $F(0)=0$ are natural assumptions in the study of free boundary problems with nonlinear operators \cite{L,SY}. In particular, they are needed for the regularity of the solution. See Theorem \ref{EvansKrylov}.

Convexity implies the existence of \textit{sub-differentials}. At a matrix $A\in\mathcal{S}_d$, a sub-differential of $F$ is a linear operator $S_A:\mathcal{S}_d\to\R$ satisfying 
\begin{equation}\label{SubDiff}
S_A(M)\le F(A+M)-F(A) \text{ for all $M\in\mathcal{S}_d$.}
\end{equation} Ellipticity \eqref{Ellipticity} implies that $S_A$ is a uniformly elliptic operator. 

 Up to a normalization, we can assume that the trace operator is a sub-differential of $F$ at $0$.\footnote{This normalization may change the ellipticity constant as in \eqref{Ellipticity}. However, the operator remains uniformly elliptic with a new ellipticity constant, depending only on the original $\Lambda.$} \end{rem}

Ellipticity \eqref{Ellipticity} and  convexity \eqref{Convexity} are sufficient for the optimal regularity of the solution (See Section 3). To study blow-up profiles, however, we need to connect solutions at different scales. This requires more assumptions on the operator $F$. 

In the literature, there are two types of natural assumptions. 

Some authors assume the operator to be differentiable \cite{SY}. Under this assumption,  we can linearize the operator. Some authors assume that the operator is homogeneous \cite{DFS, DS2}. This way,  rescalings of the solution all solve the same equation. 

In this work, we are able to deal with both  cases. That is,  we  assume that: 
{\bf Either }
\begin{equation}\label{Dable}
 F \text{ is differentiable at } 0,\footnote{Combined with \eqref{Convexity}, this implies 
 $$DF(0)(M)=trace(M).$$ Here  $DF(0)$ denotes the G\^{a}teaux derivative of $F$ at $0$ defined as 
$$DF(0)(M)=\lim_{\epsilon\to 0}{\frac{F(\epsilon M)-F(0)}{\epsilon}}.$$} 
\end{equation} {\bf or} 
\begin{equation}\label{HomogeneousF}
F \text{ is homogeneous, that is, } F(\lambda M)=\lambda F(M) \text{ for all $\lambda>0$ and $M\in\mathcal{S}_d$.} 
\end{equation}

For convex operators, solutions enjoy nice regularity. We only need the following simple version, which is a combination of interior estimates  by Caffarelli \cite{C3} and a theorem by Evans \cite{E} and Krylov \cite{K}.

\begin{thm}\label{EvansKrylov}
Let $F$ be a convex uniformly elliptic operator with ellipticity constant $\Lambda$, satisfying $F(0)=0$. Suppose that $u$ solves $$\FDu=f \textit{ in $B_1\subset\R^d$,}$$ then we have the following:

1) If $f$ is  bounded, then $u\in C^{1,\alpha}(B_{1/2})$ for all $\alpha\in(0,1)$ with 
$$\|u\|_{C^{1,\alpha}(B_{1/2})}\le C(\|u\|_{\mathcal{L}^\infty(B_1)}+\|f\|_{\mathcal{L}^\infty(B_1)})$$ for some constant $C$ depending only on $d$, $\Lambda$ and $\alpha.$

2) There is a constant $\alpha_0\in(0,1)$, depending only on $d$ and $\Lambda$, such that if $f\in C^{\alpha}(B_1)$ for some $\alpha\in(0,1)$, then $u\in C^{2,\min\{\alpha_0,\alpha\}}(B_{1/2})$ with 
$$\|u\|_{C^{2,\min\{\alpha_0,\alpha\}}(B_{1/2})}\le C(\|u\|_{\mathcal{L}^\infty(B_1)}+\|f\|_{C^\alpha(B_1)})$$ for some constant $C$ depending only on $d$, $\Lambda$ and $\alpha.$
\end{thm} 

We need the \textit{Pucci operators}, $$P^+_\Lambda,P^-_\Lambda:\mathcal{S}_d\to\R,$$defined as 
\begin{equation*}\label{Pucci}
P^+_\Lambda(M)=\sup_{\frac{1}{\Lambda}I\le A\le \Lambda I}trace(AM);\text{ and  }P^-_\Lambda(M)=\inf_{\frac{1}{\Lambda}I\le A\le \Lambda I}trace(AM).
\end{equation*} 

They are extremal among operators with ellipticity constant $\Lambda.$ In particular, if  $F(0)=0$, then 
\begin{equation}\label{CompareWithPucci}
P^+_\Lambda(M)\ge F(M).
\end{equation}

We conclude this subsection with the following  local maximum principle. See Theorem 4.8 in \cite{CC}.
\begin{thm}\label{LocMaxPple}
Suppose that $$P^+_\Lambda(D^2u)\ge0 \text{ in $B_1\subset\R^d$},$$then $$\sup_{B_{1/2}}u\le C\|u\|_{\mathcal{L}^1(B_1)}$$ for a constant $C$ depending only on $d$ and $\Lambda.$
\end{thm} 

\subsection{The \FN \AP problem}
Recall that we are considering \eqref{FirstEquation} with $F$ satisfying assumptions \eqref{Ellipticity} and \eqref{Convexity} together with \eqref{Dable} or \eqref{HomogeneousF}, and  $$\gamma\in(1,2).$$

Let $u$ be a continuous viscosity solution, then its local  boundedness  ensures that the equation is satisfied in the classical sense by Theorem \ref{EvansKrylov}.

Unlike some more standard \FB problems, the right-hand side of our equation \eqref{FirstEquation} does not jump between the \textit{contact set} $\{u=0\}$ and the \textit{non-contact set} $\PosS.$ Along the \textit{\FB}$\partial\PosS$, however, we still have a over-determined condition $$u=|\nabla u|=0 \text{ along $\partial\PosS$.}$$ This condition imposes restriction on the geometry of the free boundary. 

For notational simplicity, we introduce the following classes of solutions:
\begin{defi}\label{SolClass}
Let $\Omega$ be a domain in $\R^d$, and $\gamma\in(1,2)$. Assume that $F$ satisfies \eqref{Ellipticity} and \eqref{Convexity}, together with either \eqref{Dable} or \eqref{HomogeneousF}.

We say that  $$u\in\mathcal{S}^{F}(\Omega)$$ if $u$ solves \eqref{FirstEquation}.

Given $R>0$ and $x_0\in\R^d$, we say that $$u\in \mathcal{P}_R^{F}(x_0)$$ if $$u\in \mathcal{S}^{F}(B_R(x_0)),\text{ and } x_0\in\partial\PosS.$$ 

Finally, we say that $$u\in \mathcal{P}^{F}_\infty(x_0)$$ if $$u\in \mathcal{S}^{F}(\R^d),\text{ and } x_0\in\partial\PosS.$$ 

With a slight abuse of notation, we use $\mathcal{S}^\Delta(\Omega)$ and $\mathcal{P}_R^\Delta(x_0)$ to denote the classes when the operator is the Laplacian. 

\end{defi}

We often omit the superscript, $F$, when there is no ambiguity. 

Given $u\in\mathcal{P}_R^F(x_0)$ and $0<r<R$, we define the \textit{rescaled solution}
\begin{equation}\label{Rescaling}
u_{x_0,r}(x)=\frac{1}{r^\beta}u(rx+x_0)
\end{equation} with 
\begin{equation}\label{Beta}
\beta=\frac{2}{2-\gamma}.
\end{equation} Then we have $$u\in\mathcal{P}_{R/r}^{ F_r}(0)$$ with 
\begin{equation}\label{RescaledOperator} F_r(M)=\frac{1}{r^{\beta-2}}F(r^{\beta-2}M). \end{equation}Note that $F_r$ still satisfies assumptions \eqref{Ellipticity}, \eqref{Convexity} and \eqref{Dable} or \eqref{HomogeneousF}.

If the operator $F$ satisfies \eqref{Dable}, then we have
\begin{equation}\label{LinearizingAtZero}
\lim_{r\to 0}F_r(M)=\lim_{r\to 0}\frac{1}{r^{\beta-2}}[F(r^{\beta-2}M)-F(0)]=DF(0)M=trace(M).
\end{equation} 
If the operator satisfies \eqref{HomogeneousF}, then \begin{equation}\label{Invariance}F_r(M)=F(M) \text{ for all $r>0$.}
\end{equation}

The \FB decomposes into the regular part and the singular part according to the density of the contact set.
\begin{defi}\label{RegSing}
Suppose $u\in\mathcal{P}_R(x_0)$.

We say that $x_0$ is a \textit{regular free boundary point} if \begin{equation}\label{PositiveDensity}\limsup_{r\to 0}\frac{|B_r(x_0)\cap\{u=0\}|}{r^d}>0.\end{equation} Otherwise, $x_0$ is a \textit{singular free boundary point.}

The collection of regular free boundary points is denoted by $Reg(u)$. \end{defi}

For $u\in\mathcal{S}(\Omega)$, we define the \textit{linearized operator} as 
\begin{equation}\label{LinearizedOperator}
L_u(w)=S_{D^2u}(D^2 w)-(\gamma-1)u^{\gamma-2}w \text{ in $\PosS\cap\Omega$.}
\end{equation}  
Here $S_{D^2u}$ is a sub-differential of $F$ at $D^2u$ as in \eqref{SubDiff}.

For a direction $e\in\Sph$, where $\Sph$ denotes the unit sphere,  let $D_eu$ and $D_{ee}u$ denote, respectively, the directional derivative and second derivative of $u$ in the $e$-direction. These are both super-solutions to the linearized operator. 

\begin{prop}\label{DerEquation}For $u\in\mathcal{S}(\Omega)$, we have
\begin{equation*}L_u(D_eu)\leq0 \text{ and }L_u(D_{ee}u)\leq0\text{ in $\PosS\cap\Omega$.}\end{equation*} 
\end{prop} 
\begin{proof}
For a positive real number $t>0$, the definition of sub-differentials \eqref{SubDiff} gives 
$$\frac{F(D^2u(x+te))-F(D^2u(x))}{t}\ge S_{D^2u(x)}(\frac{D^2u(x+te)-D^2u(x)}{t}).$$Sending $t\to0$ and using \eqref{FirstEquation}, we get the first inequality. 

Similarly,  we have  
\begin{align*}&\frac{F(D^2u(x+te))+F(D^2u(x-te))-2F(D^2u(x))}{t^2}\\\ge &S_{D^2u(x)}(\frac{D^2u(x+te)+D^2u(x-te)-2D^2u(x)}{t^2}).\end{align*}
Sending $t\to0$ gives
$$S_{D^2u} (D_{ee}u)\le(\gamma-1)u^{\gamma-2}D_{ee}u +(\gamma-1)(\gamma-2)u^{\gamma-3}(D_eu)^2 \text{ in $\{u>0\}\cap\Omega$.}$$
 
  With $\gamma\in(1,2),$ the last term is non-positive, this leads to the second inequality.
  \end{proof}

As a direct consequence of \eqref{Convexity}, we have the following 
\begin{prop}\label{CompareWithLinear}
Suppose $u\in\mathcal{S}(\Omega)$, then we have $$L_u(u)\ge(2-\gamma)u^{\gamma-1} \text{ in $\PosS\cap\Omega$, }$$ and $$\Delta u\le \uga \text{ in $\Omega$.}$$
\end{prop} 

\begin{proof}
By \eqref{Convexity}, we have $$S_{D^2u}(D^2u)\ge \FDu-F(0)\ge  trace(D^2u),$$  where $S_{D^2u}$ is a sub-differential of $F$ at $D^2u$ defined in \eqref{SubDiff}.

With this, the comparisons follow from $F(0)=0$ and $\FDu=\uga$. 
\end{proof} 

This gives us a useful barrier.
\begin{lem}\label{Barrier}
Suppose that $u\in\mathcal{S}(\Omega)$ and $x_0\in\R^d$. Let $\beta$ be as in \eqref{Beta}.

  There is a universal constant $a_0>0$ such that for all $a\ge a_0$, we have $$L_u(au-|x-x_0|^\beta)\ge 0 \text{ in $\PosS\cap\Omega.$}$$
\end{lem} 
Recall that if a constant depends only on the dimension $d$, the ellipticity $\Lambda$ and the parameter $\gamma$, then it is called a \textit{universal} constant.

\begin{proof}To simplify our notation, let's define $\rho=|x-x_0|$.

Using the homogeneity of $\rho^\beta$ and ellipticity \eqref{Ellipticity}, we have 
$$S_{D^2u}(D^2\rho^{\beta})\le C\rho^{\beta-2}$$for a universal constant $C$. As a result, we have 
$$L_u(\rho^\beta)\le C\rho^{\beta-2}-(\gamma-1)u^{\gamma-2}\rho^{\beta} \text{ in $\PosS\cap\Omega.$}$$
Note that $L_u(u)\ge 0$, we have the desired estimate for any $a\ge0$   in $\Omega\cap\PosS\cap\{\frac{C}{\gamma-1}u^{2-\gamma}\le\rho^2\}$. 

On the other hand, inside $\{\frac{C}{\gamma-1}u^{2-\gamma}>\rho^2\}=\{\rho<Cu^{\frac{2-\gamma}{2}}\}$, Proposition \ref{CompareWithLinear} gives
\begin{align*} 
L_u(au-\rho^\beta)&\ge a(2-\gamma)\uga-C\rho^{\beta-2}+(\gamma-1)u^{\gamma-2}\rho^\beta\\&\ge a(2-\gamma)\uga-Cu^{\frac{2-\gamma}{2}(\beta-2)}+(\gamma-1)u^{\gamma-2}u^{\frac{2-\gamma}{2}\beta}\\&=\uga(a(2-\gamma)+\gamma-1-C)
\end{align*} for a universal constant $C$.

With $\gamma\in(1,2)$, we see the last term is non-negative if $a$ is universally large. 
\end{proof} 

As a first consequence of this barrier, we have the \textit{non-degeneracy} of the solution. 

\begin{cor}\label{NonDeg}
For $u\in \mathcal{P}_R(x_0)$, we have $$\sup_{\partial B_r(x_0)}u\ge Cr^\beta$$ for a universal constant $C$ for all $0<r<R.$ 
\end{cor} 

\begin{proof}
Pick $y_0\in\PosS$ with $|x_0-y_0|<\frac{1}{2}r$.
We let $\rho=|x-y_0|$.

For $a_0$ as in Lemma \ref{Barrier}, define 
$$h=a_0u-\rho^\beta$$ inside the domain $\Omega=B_r(x_0)\cap\{u>0\}.$ 

With $h(y_0)>0$ and $L_u(u)\ge 0$ in $\Omega$,  the maximum principle implies $\sup_{\partial\Omega}h> 0.$

Note that $$\partial\Omega=(\partial\PosS\cap B_r(x_0))\cup (\partial B_r(x_0)\cap\{u>0\}),$$ and that $h\le 0$ along $\partial\PosS\cap B_r(x_0)$, this implies $$\sup_{\partial B_r(x_0)\cap\{u>0\}}h=\sup_{\partial\Omega} h>0.$$The desired estimate follows from $\rho>\frac{1}{2}r$ along $\partial B_r(x_0)$.
\end{proof} 

The improvement of monotonicity is a useful technique in many \FB problems. We prove a version for our problem.
\begin{lem}\label{ImproveMonotonicity}
Suppose $u\in\mathcal{S}(B_1)$.   Let $a_0$ denote the constant from Lemma \ref{Barrier} and $\beta$ be the parameter in \eqref{Beta}.

Suppose that for some direction $e\in\Sph$, and some constants $\kappa>0$ and $0<\eta\le\frac{1}{2^{\beta+1}a_0}$, we have $$D_eu\ge\kappa \text{ in $B_1\cap\{u\ge\eta\}$ }.$$ Then there is $\eps>0$, depending only on $\kappa$, such that $$D_eu\ge-\eps \text{ in $B_1$}$$ implies $$D_eu\ge 0 \text{ in $B_{1/2}$.}$$
\end{lem}

\begin{proof}
We  choose $\eps=\kappa$. 

Note that we have $D_eu\ge 0$ inside $\{u\ge\eta\}$. Since $u\ge0$, we also have $D_eu=0$ in $\{u=0\}.$ It suffices to prove $D_eu\ge0$ in $B_{1/2}\cap\{0<u<\eta\}.$

To see this, we pick a point $x_0\in B_{1/2}\cap\{0<u<\eta\}.$ Define $$\Omega=B_1\cap\{0<u<\eta\},$$ $$\rho=|x-x_0|,$$  and $$h=D_eu-2^{\beta+1}\kappa(a_0u-\rho^\beta).$$

We first investigate the value of $h$ along $\partial\Omega=(\partial B_1\cap \{0<u<\eta\})\cup (B_1\cap\partial\PosS)\cup (B_1\cap\partial\{u<\eta\}).$

Along  $\partial B_1\cap \{0<u<\eta\}$, we have $\rho\ge\frac{1}{2}$ and $u<\eta$, thus $$h\ge-\eps-2^{\beta+1}\kappa(a_0\eta-\frac{1}{2^\beta}).$$ Since $\eta\le\frac{1}{2^{\beta+1}a_0}$ and $\eps=\kappa$, we have $h\ge 0$ along $\partial B_1\cap \{0<u<\eta\}$.

Along  $B_1\cap\partial\PosS,$ we have $u=D_eu=0$, thus $h\ge 0.$

Along $B_1\cap\partial\{u<\eta\}$, we have $u=\eta$ and $D_eu\ge\kappa$. Consequently,  
$$h\ge\kappa-2^{\beta+1}\kappa(a_0\eta)\ge0.$$In the last step we used the assumption $\eta\le\frac{1}{2^{\beta+1}a_0}$.

To summarize, $$h\ge 0\text{ along $\partial \Omega.$}$$

 Together with Lemma \ref{Barrier},  Proposition \ref{DerEquation} implies $L_u(h)\le 0$ in $\Omega$. It then follows from the maximum principle that $$h\ge 0 \text{ in $\Omega$.}$$ At the point $x_0$, this  implies $D_eu(x_0)\ge 0.$
\end{proof}

\section{Harnack estimate and consequences}
Here we prove the Harnack-type estimate, Theorem \ref{HarnackThm}. With this, we get the universal regularity estimate in Theorem \ref{FirstResult}, as well as control over several scaling invariant quantities. These are useful to study the blow-up profiles.

Theorem \ref{HarnackThm} follows directly from the following lemma.
\begin{lem}\label{HarnackLem}
Suppose that $u\in\mathcal{S}(B_1(0))$. Then $$\sup_{B_{1/2}}u\le C(u(0)+1)$$ for a universal constant $C$.
\end{lem} 
Recall that a constant is universal if it depends only on $d$, $\Lambda$, and $\gamma$. Also recall the definition of several classes of solutions from Definition \ref{SolClass}.

\begin{proof}
We divide the proof into three steps. 

In the first step, we prove a supersolution property for the classical \AP equation. This step is essentially from \cite{P}, and is recorded here for completeness. In the second step, we show that our solution inherits this supersolution property. In the third step, we conclude the proof by using a subsolution property. 

\textit{Step 1: the supersolution property for the classical equation.}

Let $v$ be the solution to 
$$\begin{cases}
\Delta v=v^{\gamma-1}&\text{ in $B_1$,}\\
v=u &\text{ on $\partial B_1.$}
\end{cases}
$$In this step, we  establish the following claim:

\textit{Claim}: There is a universal constant, $A$, such that $$\text{if }\avint_{\partial B_1}v\ge A,\text{ then }v(0)\ge \frac{1}{2}\avint_{\partial B_1}v.$$Here we denote by $\avint_{\partial B_r}v$ the average of $v$ over $\partial B_r.$

With $u\ge0$,  the maximum principle implies $v\ge 0$  in $B_1$. In particular, $$\Delta v\ge0 \text{ in $B_1$,}$$ and $\avint_{\partial B_r}v$  is an increasing function of $r$. 

Let $G$ denote the Green's function of $B_1$ with a pole at $0$, that is, $G(x)=\log|x|$ if $d=2$ and $G(x)=-|x|^{2-d}+1$ if $d\ge3.$ We have
$$\avint_{\partial B_1}v-v(0)=-C\int_{B_1}G\Delta v=-C\int_{B_1}Gv^{\gamma-1}=-C\int_0^1dr\int_{\partial B_r}Gv^{\gamma-1}$$for a positive dimensional constant $C$. 

Consequently, we can use H\"older's inequality to get 
\begin{align*}
\avint_{\partial B_1}v-v(0)&\le C\int_0^1dr[\int_{\partial B_r}(r^{(d-1)(\gamma-1)}|G|)^{\frac{1}{2-\gamma}}]^{2-\gamma}\cdot(\avint_{\partial B_r}v)^{\gamma-1}\\&\le C(\avint_{\partial B_1}v)^{\gamma-1}\\&=C(\avint_{\partial B_1}v)^{\gamma-2}\avint_{\partial B_1}v.
\end{align*}

Since $\gamma-2<0,$
 the right-hand side is less than $\frac{1}{2}\avint_{\partial B_1}v$ if $A$ is large. This gives the \textit{Claim}. 

\textit{Step 2: the supersolution property for the fully nonlinear equation.}

In this step, we prove that for our solution $u$ to the \FN problem, we have 

\textit{Claim}: There is a positive universal constant $A$ such that $$\text{if }\avint_{\partial B_1}u\ge A,\text{ then }u(0)\ge \frac{1}{2}\avint_{\partial B_1}u.$$

Let $v$ be the solution from \textit{Step 1}. It suffices to prove $$u\ge v\text{ in $B_1$}.$$
To this end, let $t_0=\min_{\overline{B}_1}(u-v)=(u-v)(x_0)$ for some $x_0\in\overline{B}_1.$ We need to show that $t_0\ge 0.$

If $x_0\in\partial B_1$, then $t_0=0$. Otherwise, $x_0\in B_1$ is a local minimum of $(u-v)$, and we have $$\Delta u(x_0) \ge\Delta v(x_0). $$The second part of Proposition \ref{CompareWithLinear} gives $$\uga(x_0)\ge\Delta u(x_0)\ge\Delta v(x_0)=v^{\gamma-1}(x_0).$$Together with $u(x_0)-v(x_0)=t_0$, this implies $t_0\ge 0.$

\textit{Step 3: Conclusion of the proof.}

Using the scaling property as in \eqref{Rescaling}, \textit{Step 2} gives, for all $0<r<1$, $$\text{ either }u(0)\ge \frac{1}{2}\avint_{\partial B_r}u \text{, or } \avint_{\partial B_r}u\le Ar^\beta.$$ Consequently,  $$\int_{B_1}u=C(d)\int_0^1dr\avint_{\partial B_r}r^{d-1}u\le C(d)(\int_{0}^1Ar^{d-1+\beta}dr+\int_0^12u(0)r^{d-1}dr)\le C(u(0)+1)$$ for a universal $C$. 

Now with $P^+_\Lambda(D^2u)\ge \FDu=\uga\ge 0$ as in \eqref{CompareWithPucci}, we can invoke Theorem \ref{LocMaxPple} to conclude this proof. 
\end{proof} 

Theorem \ref{HarnackThm} follows from Lemma \ref{HarnackLem} by scaling. Theorem \ref{FirstResult} is also a direct consequence. We briefly sketch the ideas behind it. 

\begin{proof}[Sketch of the proof of Theorem \ref{FirstResult}]
Let $u$ be as in the statement of Theorem \ref{FirstResult}. 

Theorem \ref{HarnackThm} implies that $u$ is  locally bounded. This implies that the right-hand side of \eqref{FirstEquation}, $\uga$,  is locally  bounded. Using the first part of Theorem \ref{EvansKrylov}, we have $u\in C_{loc}^{1,\alpha}$. As a result, the right-hand side, $\uga$, is H\"older continuous. We can then use the second part of Theorem \ref{EvansKrylov} to get $u\in C^{2,\alpha}.$

To get the bound for $\|u\|_{\mathcal{C}^{2,\alpha}(B_{1/2})}$ if $u$ solves the problem in $B_1$, we can use Lemma \ref{HarnackLem} to get $u\le C(u(0)+1)$ in $B_{7/8}.$ This gives $\|u\|_{C^{1,\alpha}(B_{3/4})}\le C(u(0)+1)$ by invoking the first part of Theorem \ref{EvansKrylov}. The second part of Theorem \ref{EvansKrylov} then gives the desired estimate. \end{proof} 

We gather a few quick consequences of Theorem \ref{HarnackThm}.

\begin{cor}\label{GrowthRate}
Suppose $u\in \mathcal{P}_R(x_0)$. Then $$\sup_{B_r(x_0)}u\le Cr^\beta$$ for all $0<r<\frac{1}{2}R$ for a universal constant $C$.
\end{cor} 
The corresponding lower bound is proved in Corollary \ref{NonDeg}. 

A similar estimate was obtained by Teixeira in a different context \cite{T}. His estimate, however, depends on the $\mathcal{L}^\infty$-norm of $u$, while ours is universal. 

We also have a universal bound on the following scaling invariant quantity.
\begin{cor}\label{HessianBound}
Suppose $u\in \mathcal{P}_R(x_0)$. Then 
$$\sup_{B_{R/2}(x_0)\cap\PosS}|D^2u/\uga|\le C$$ for a universal $C$.
\end{cor}
\begin{proof}
It suffices to prove the estimate for $x_0=0$ and $R=1$.

For any $y_0\in B_{1/2}\cap\PosS$, Lemma \ref{HarnackLem} implies that we can find a universal constant, $M$, such that $$r:=(\frac{u(y_0)}{M})^{\frac{1}{\beta}}<\frac{1}{2}.$$ In particular, $u\in \mathcal{S}(B_r(y_0)).$

Consequently, for $u_{y_0,r}$ defined as in \eqref{Rescaling}, we have $u_{y_0,r}\in \mathcal{S}^{F_r}(B_1(0))$ for $F_r$ defined as in \eqref{RescaledOperator}. 

Lemma \ref{HarnackLem} implies $$\sup_{B_{1/2}}u_{y_0,r}\le C(u_{y_0,r}(0)+1)=C(M+1).$$ 

We can apply the first part of Theorem \ref{EvansKrylov} with $f=u_{y_0,r}^{\gamma-1}$ to get $|\nabla u_{y_0,r}|\le C$ in $B_{1/4}.$ This then implies the $(\gamma-1)$-H\"older semi-norm of $f$ satisfies $[f]_{C^{\gamma-1}(B_{1/4})}\le C$. We can  apply the second half of Theorem \ref{EvansKrylov} to get $|D^2u_{y_0,r}|(0)\le C$.

By the definition of $r$, this gives $|D^2u(y_0)/\uga(y_0)|\le C.$
\end{proof}

\section{Blow-up profile at a regular point}
In this section, we begin our analysis of the \FB  near a regular point. The first step is to study the blow-up profiles. 

As mentioned in the introduction, without monotonicity formulae, blow-ups are in general not homogeneous,  their contact sets not necessarily conic. It is challenging to classify blow-up profiles. Here we intend to get  geometric information in terms of monotonicity properties of blow-ups. 

Suppose $u\in\mathcal{P}^F_1(0)$ with $0\in Reg(u)$ (Definitions \ref{SolClass} and \ref{RegSing}) for some $F$ satisfying \eqref{Ellipticity} and \eqref{Convexity} together with either \eqref{Dable} or \eqref{HomogeneousF}.

The rescalings $u_r(x)=\frac{1}{r^\beta}u(rx)$ as in \eqref{Rescaling} satisfy $$u_r\in\mathcal{P}_{1/r}^{F_r}(0),$$ where $F_r$ is defined as in \eqref{RescaledOperator}. Results from the previous section imply that this family $\{u_r\}$ is locally uniformly  bounded in $C^{2,\alpha}.$ As a result, we have a function $u_0$, such that up to a subsequence, $$u_r\to u_0 \text{ locally uniformly in $C^2(\R^d)$.}$$

For $F$ satisfying \eqref{Dable},  following Theorem \ref{Stability} and Corollary \ref{NonDeg}, we can use \eqref{LinearizingAtZero} to get $u_0\in\mathcal{P}_\infty^\Delta(0).$ 
If $F$ satisfies  \eqref{HomogeneousF}, we use \eqref{Invariance} to get $u_0\in\mathcal{P}_\infty^F(0).$ 

Note that the Laplacian is  homogeneous, we see that under both cases, we have $$u_0\in\mathcal{P}_\infty^F(0)$$ for some $F$ satisfying \eqref{Ellipticity}, \eqref{Convexity} and \eqref{HomogeneousF}.

Positive density of the contact set, \eqref{PositiveDensity}, gives some $\delta>0$ and a sequence $r_h\to0$ such that $$|\{u_{r_h}=0\}\cap B_1|\ge\delta.$$ If we perform the procedure above with this particular sequence $r_h$, then we have $$|\{u_0=0\}|>0.$$

To sum up, we obtain a blow-up limit satisfying the hypothesis in the following proposition, which is the main result of this section.

\begin{prop}\label{BlowUp}
Suppose that $u\in\mathcal{P}^F_\infty(0)$ with $|\ZeroS|>0$ for some $F$ satisfying \eqref{Ellipticity}, \eqref{Convexity} and \eqref{HomogeneousF}. 

Then up to a rotation, we have the following:
 
Given $\delta>0$, there is $r=r_\delta>0$ such that 
\begin{equation}\label{BlowUpMonotonicity}\begin{cases}D_eu\ge 0 &\text{ in $B_{r}$,}\\ D_eu\ge c_0\delta r^{\beta-1} &\text{ in $B_{r}\cap\{u\ge\frac{1}{2^{\beta+1}a_0}r^\beta$\}}
\end{cases}
\end{equation}for all $e\in\Sph$ with $e_1\ge\delta$. 
Here $c_0$ is a universal constant.\end{prop}

The constants  $a_0$ and $\beta$ come from Lemma \ref{Barrier} and \eqref{Beta}, respectively.
For a vector $x\in\R^d$, its first coordinate in the standard basis is denoted by $x_1.$

For the solution class $\mathcal{P}^F_\infty(0)$, see Definition \ref{SolClass}. 

We first establish several lemmata concerning global solutions. 

\begin{lem}\label{GlobConv}
Suppose that $u\in\mathcal{P}^F_\infty(0)$  for some $F$ satisfying \eqref{Ellipticity}, \eqref{Convexity} and \eqref{HomogeneousF}. Then $u$ is convex. 
\end{lem} 
\begin{proof}
We prove that the pure second derivative of $u$ in the $x_1$-direction is non-negative, that is, $$D_{11}u\ge0 \text{ in $\R^d.$}$$ With $u\ge 0$ in $\R^d$, it suffices to prove this in $\PosS$.

Suppose the desired estimate is not true, then we have 
\begin{equation}\label{Ell}-\ell:=\inf_{\PosS}\frac{D_{11}u}{\uga}<0.\end{equation} This infimum is well-defined as $|D^2u/\uga|$ is universally  bounded in $\PosS$ by Corollary \ref{HessianBound}.

We seek a contradiction in several steps. 

In the first step, we define a family of normalized solutions, centered at `worst' points for the ratio in \eqref{Ell}. This family is compact and converges to a limit. In the second step, we compute this ratio at the limiting function, and show that it achieves an interior minimum. In the third step, we get a contradiction by considering the equation for the limit. 

\textit{Step 1: Normalized solutions and compactness.}

Let $x_h\in\PosS$ be a minimizing sequence in the sense that 
\begin{equation}\label{Minimizing}
\frac{D_{11}u}{\uga}(x_h)\to-\ell.
\end{equation} Take $$r_h=u^{\frac{1}{\beta}}(x_h),$$  we define the normalized solutions as $$v_h(x)=\frac{1}{r_h^\beta}u(r_hx+x_h).$$

Homogeneity of $F$ implies that $v_h\in\mathcal{S}^F(B_2(0))$ with $v_h(0)=1$ for all $h$. Theorem \ref{FirstResult} implies that $\{v_h\}$ is uniformly bounded in $C^{2,\alpha}(B_1)$. Up to a subsequence, we have $$v_h\to w \text{ in $C^2(B_1)$}$$for some $$w\in\mathcal{S}^F(B_1)$$ and $w(0)=1.$

\textit{Step 2: An interior minimum.}

Note that $$D_{11}v_h(x)=\frac{1}{r_h^{\beta-2}}D_{11}u(r_hx+x_h),$$ we have, by \eqref{Ell}, 
$$D_{11}v_h(x)\ge\frac{1}{r_h^{\beta-2}}(-\ell u^{\gamma-1}(r_hx+x_h))=-\ell v^{\gamma-1}_h(x)$$ for all $x\in\R^d.$ Passing to the limit gives $$D_{11}w\ge-\ell w^{\gamma-1} \text{ in $B_1.$}$$

On the other hand, $$D_{11}v_h(0)=\frac{1}{r_h^{\beta-2}}D_{11}u(x_h)=\frac{D_{11}u(x_h)}{u(x_h)^{\gamma-1}}$$ by the definition of $r_h.$ Thus \eqref{Minimizing} gives $$D_{11}w(0)=-\ell=-\ell w^{\gamma-1}(0).$$ 

Define $$g=D_{11}w+\ell w^{\gamma-1},$$ then $g\ge 0$ in $B_1$ and $g(0)=0.$

\textit{Step 3: The contradiction.}

A direct computation gives 
$$D^2(w^{\gamma-1})=(\gamma-1)w^{\gamma-2}D^2w+(\gamma-1)(\gamma-2)w^{\gamma-3}\nabla w\otimes\nabla w,$$where $\otimes$ denotes the tensor product. With $\gamma\in(1,2)$, the last term is non-positive. 

Recall that the sub-differential $S_{D^2w}$ in \eqref{SubDiff}  is elliptic, we have 
$$S_{D^2w}(D^2(w^{\gamma-1}))\le (\gamma-1)w^{\gamma-2}S_{D^2w}(D^2w) \text{ in $B_1\cap\{w>0\}.$}$$
Now we use the definition of sub-differentials, the homogeneity of $F$    and \eqref{FirstEquation} to get 
$$ S_{D^2w}(D^2w)\le F(2D^2w)-F(D^2w)=F(D^2w)=w^{\gamma-1}.$$Combine this with the previous estimate, we have $$L_w(w^{\gamma-1})\le 0 \text{ in $B_1\cap\{w>0\},$}$$where $L_w$ is the linearized operator from \eqref{LinearizedOperator}.

Now note that Proposition \ref{DerEquation} implies $$L_w(D_{11}w)\le 0\text{ in $B_1\cap\{w>0\}.$}$$We can use $\ell>0$ to get $$L_w(g)\le 0\text{ in $B_1\cap\{w>0\}.$}$$

 This contradicts the strong maximum principle since $g$ achieves its minimum at an interior point of $B_1\cap \{w>0\}.$
\end{proof} 

The next lemma classifies global solutions with nontrivial \textit{conic} contact sets. 
\begin{lem}\label{HSSolution}
Let $u$ be as in the previous lemma. 

If the contact set $\ZeroS$ is a cone with $|\ZeroS|>0$, then 
 $$u(x)=c_{\gamma,e} [(x\cdot e)_+]^\beta,$$ for some $e\in \Sph$ and $c_{\gamma,e}\in\R$ satisfying $0<c\le c_{\gamma,e}\le C$ for some universal constants $c$ and $C$.
\end{lem} Recall that for a real number $a$, we denote by $a_+$ its positive part. 

\begin{proof}
By Lemma \ref{GlobConv}, we know that $u$ and $\ZeroS$ are convex. 

We first establish the following property of the contact set:

\textit{Claim:} For each $-e\in\Sph\cap\partial\ZeroS$, we have $e\in\ZeroS.$

With $-e\in\ZeroS$ and $\ZeroS$ being a cone, we have $$u(-te)=0 \text{ for all $t>0.$}$$ Since $u\ge0$   and $u$ is convex, it is elementary that   $$D_eu\ge 0 \text{ in $\R^d$.}$$ 

With Proposition \ref{DerEquation},  there are only two options, by the strong maximum principle: \begin{equation}\label{AorB}
\text{either } D_eu=0 \text{ in } \R^d, \text{ or } D_eu>0 \text{ in } \PosS.
\end{equation}  

In the first case, the \textit{Claim} holds trivially. We consider the second case. 

In this case,  continuity of $D_eu$ implies $$D_eu\ge2\kappa \text{ in $B_2\cap\{u\ge\frac{1}{2^{\beta+1}a_0}\}$}$$ for some $\kappa>0.$ Here $a_0$ is the constant from Lemma \ref{Barrier}. 

Consequently, we can find $\delta>0$, depending on $\kappa$, such that $D_\nu u$ satisfies the hypothesis in Lemma \ref{ImproveMonotonicity} whenever $|\nu-e|<\delta$. It follows that $$D_\nu u\ge 0 \text{ in $B_1$.}$$With $0\in\ZeroS$, this implies $$-\nu\in\ZeroS\text{ for all }|\nu-e|<\delta,$$ contradicting the assumption that $-e\in\partial\ZeroS.$

Therefore, the second case in \eqref{AorB} cannot happen. This completes the proof of the \textit{Claim}.

Now that $\ZeroS$ is a nontrivial convex cone satisfying the \textit{Claim}, it must be a half-space, that is $$\ZeroS=\{x\cdot e\le  0\}$$ for some $e\in\Sph.$

In particular, we have $$\{x\cdot e=0\}\subset\ZeroS.$$ This implies, by convexity of $u$, that $u$ only depends on $x\cdot e.$

Now an ODE argument gives that $u$ is of the form $u=c_{\gamma,e} [(x\cdot e)_+]^\beta,$ with a constant $c_{\gamma,e}$ satisfying 
$c^{2-\gamma}\beta(\beta-1)F(e\otimes e)=1.$ Using ellipticity of $F$, we see that this is  bounded away from $0$ and infinity by universal constants. 
 \end{proof} 

With these preparations, we give 
\begin{proof}[Proof of Proposition \ref{BlowUp}]
Take a sequence $r_h\to0$, we define the rescalings $$u_h(x)=\frac{1}{r_h^\beta}u(r_hx).$$ Up to a subsequence, $$u_h\to v \text{ locally uniformly in $C^2(\R^d)$}$$ for some $v\in\mathcal{P}^F_\infty(0).$

By Lemma \ref{GlobConv}, the contact set $\ZeroS$ is a nontrival convex set. It follows that $\{v=0\}$ is a convex cone. 

Lemma \ref{HSSolution} implies, up to a rotation, $$v=c_\gamma [(x_1)_+]^\beta.$$

This function satisfies $$D_ev\ge 0 \text{ for all $e\in\Sph$ with $e_1\ge 0$.}$$  Moreover, there is a universal constant, $c>0$, such that $$D_ev\ge 2c\delta\text{ in $\{v\ge\frac{1}{2^{\beta+2}a_0}\}$ whenever $e_1\ge\delta.$}$$

By the convergence of $u_h\to v$, we have, for large $h$, 
$$\{u_h\ge\frac{1}{2^{\beta+1}a_0}\}\cap B_2\subset \{v\ge\frac{1}{2^{\beta+2}a_0}\}.$$ Together with the convergence of derivatives, we have, for large $h$,  $$D_eu_h\ge c\delta \text{ in $\{u_h\ge\frac{1}{2^{\beta+1}a_0}\}\cap B_2$}$$ whenever $e_1\ge\delta.$

Since $D_eu_h\to D_ev$ uniformly in $B_2$ and $D_ev\ge 0$ for all $e_1\ge 0$, we can invoke Lemma \ref{ImproveMonotonicity} to get, for large $h$,  $$D_eu_h \ge 0\text{ in $B_1$ whenever $e_1\ge\delta.$}$$

In summary, we have shown that, for large $h$, 
$$D_eu_h\ge c\delta \text{ in $\{u_h\ge\frac{1}{2^{\beta+1}a_0}\}\cap B_1$}$$ and $$D_eu_h\ge 0 \text{ in $B_1$}$$whenever $e_1\ge\delta.$

Fix one such large $h$, say $h_0$, we see that the desired comparisons hold for $u$ in $B_{r_{h_0}}.$\end{proof}

\section{Regularity of the free boundary}
In this final section, we give the proof of Theorem \ref{SecondResult}. 

We first prove the relative openness of $Reg(u)$ in $\partial\PosS$, which is the content of the following result:
\begin{prop}\label{RelativeOpen}Let $F$ be an operator satisfying \eqref{Ellipticity} and \eqref{Convexity} together with either \eqref{Dable} or \eqref{HomogeneousF}.

Suppose that $u\in \mathcal{P}^F_1(0)$ with $0\in Reg(u).$ 

Then there is $\rho>0$ such that $B_\rho\cap\partial\PosS\subset Reg(u).$
\end{prop} For the definition of the solution class and of the regular part, see Definitions \ref{SolClass} and \ref{RegSing}.

\begin{proof}
As we remarked before Proposition \ref{BlowUp}, we can find a sequence $r_h\to 0$, such that the the rescalings $$u_{h}(x)=\frac{1}{r_h^\beta}u(r_hx)\to u_0 \text{ locally uniformly in $C^2(\R^d)$}$$ for some $u_0$ as in Proposition \ref{BlowUp}.

In particular, by taking $\delta=\frac{1}{2}$ as in Proposition \ref{BlowUp}, we find $r>0$ such that 
\begin{equation*}\begin{cases}D_eu_0\ge 0 &\text{ in $B_{r}$,}\\ D_eu_0\ge \frac{1}{2}c_0 r^{\beta-1} &\text{ in $B_{r}\cap\{u_0\ge\frac{1}{2^{\beta+1}a_0}r^\beta$\}}
\end{cases}
\end{equation*}for all $e\in\Sph$ with $e_1\ge 1/2$.

For $\eps>0$ to be chosen, locally uniform $C^2$ convergence then gives, for large $h$, 
\begin{equation*}\begin{cases}D_eu_h\ge -\eps &\text{ in $B_{r}$,}\\ D_eu_h\ge \frac{1}{4}c_0 r^{\beta-1} &\text{ in $B_{r}\cap\{u_h\ge\frac{1}{2^{\beta}a_0}r^\beta$\}}
\end{cases}
\end{equation*}for all $e\in\Sph$ with $e_1\ge 1/2$. 

If we choose $\eps$, depending on $c_0$, according to Lemma \ref{ImproveMonotonicity}, this gives 
$$D_eu_h\ge 0 \text{ in $B_{\frac{1}{2}r}$}$$ for all $e_1\ge 1/2$ if $h$ is large. 

Fix such a large $h_0.$
Scale back to $u$, this gives a positive $\rho=\frac{1}{2}r_{h_0}r>0$ such that $$D_eu\ge 0 \text{ in $B_{\rho}$}$$ for all $e_1\ge 1/2$.

It is now elementary to see that there is a function $$f: B_\rho\cap\{x_1=0\}\to\R,$$ with universal Lipschitz norm such that $$\ZeroS\cap B_\rho=\{x_1\le f(x')\} \text{ and } \PosS\cap B_{\rho}=\{x_1> f(x')\}.$$

Since the contact set $\ZeroS$ in $B_\rho$ is the region below a Lipschitz graph, all \FB points in $B_\rho$ satisfy \eqref{PositiveDensity}.
\end{proof} 

Now we show that $Reg(u)$ is locally a $C^1$-hypersurface.

\begin{prop}
Under the same assumptions as in the previous proposition, we can find $\rho>0$ such that $B_\rho\cap\partial\PosS$ is a $C^1$-hypersurface.
\end{prop} 
\begin{proof}
By using the same argument as in the previous proof, but with a general $\delta>0$ instead of $\delta=\frac{1}{2}$, we have the following:
\begin{equation}\label{MonProp}\text{Given }\delta>0, \text{ there is }\rho_\delta>0 \text{ such that }D_eu\ge 0 \text{ in } B_{\rho_\delta} \text{ for all $e_1\ge\delta.$}\end{equation}

Recall from the proof of the previous proposition, we have that $\partial\PosS$ is the graph of some function $f$. Property \eqref{MonProp} implies that $|f(x')-f(0)|\le\delta|x'|$ for all $|x'|<\rho_\delta.$
By sending $\delta\to0$, this shows that the function $f$ is differentiable at $0$ with $0$ derivative. In terms of the free boundary, this implies that $\partial\PosS$ is differentiable at $0$ with normal $\xi_1$, the unit vector in the $x_1$-direction.

By using this argument at different points, this shows that $\partial\PosS$ is a differentiable hypersurface near $0$, with a well-defined normal $\nu_x$ at each $x\in\partial\PosS.$

Again by the monotonicity property \eqref{MonProp}, we see that inside $B_{\rho_\delta}$, $$|\nu_x-\xi_1|\le C\delta.$$By sending $\delta\to 0,$ this gives the continuity of normals at $0$. The same argument can be applied to all nearby \FB points. 
\end{proof}

\section*{Acknowledgement} HY would like to thank Connor Mooney for sharing his insight about boundary Harnack principles during the preparation of this manuscript.  Both authors would like to thank two referees for the careful reading of the first draft and for providing many insightful suggestions.


\end{document}